\documentclass[12pt,english]{article}
\usepackage[T1]{fontenc}
\usepackage[latin9]{inputenc}
\usepackage{geometry}
\usepackage{amsmath}
\usepackage{amsthm}
\usepackage{amssymb}

\makeatletter

\providecommand{\tabularnewline}{\\}

\theoremstyle{plain}
\newtheorem{thm}{\protect\theoremname}[section]
\theoremstyle{plain}
\newtheorem{lem}[thm]{\protect\lemmaname}
\theoremstyle{remark}
\newtheorem{rem}[thm]{\protect\remarkname}

\makeatother

\usepackage{babel}
\providecommand{\lemmaname}{Lemma}
\providecommand{\remarkname}{Remark}
\providecommand{\theoremname}{Theorem}

\begin{document}
\baselineskip .275in

\title{Analytical Results on the Service Performance of Stochastic Clearing
Systems}
\author{$\text{Bo Wei}$$\qquad$$\text{S{\i}la \c{C}etinkaya}$$\qquad$ Daren
B.H. Cline}
\date{}
\maketitle
\begin{abstract}
\noindent Stochastic clearing theory has wide spread applications in the context of supply chain and service operations management. Historical application domains include bulk service queues, inventory control, and transportation planning (e.g., vehicle dispatching and shipment consolidation).  In this paper, motivated by a fundamental application in shipment consolidation, we revisit the notion of service performance for stochastic clearing system operation. More specifically, our goal is to evaluate and compare service performance of alternative operational policies for clearing decisions, as quantified by a measure of timely service referred as \emph{Average Order Delay} ($AOD$). All stochastic clearing systems are subject to service delay due to the inherent clearing practice, and $AOD$ can be thought of as a benchmark for evaluating timely service. Although stochastic clearing theory has a long history, existing literature on the analysis of $AOD$ as a service measure has several limitations. Hence, we extend the previous analysis by proposing a more general method for a generic analytical derivation of $AOD$ for any renewal-type clearing policy, including but not limited to alternative shipment consolidation policies in the previous literature. Our proposed method utilizes a new martingale point of view and lends itself for a generic analytical characterization of $AOD$, leading to a complete comparative analysis of alternative renewal-type clearing policies. Hence, we also close the gaps in literature on shipment consolidation via a complete set of analytically provable results regarding $AOD$ which were only illustrated through numerical tests previously. \\


\end{abstract}

\newpage 
\section{Introduction \& Motivation \label{sec:Introduction}}

\emph{"Stochastic clearing systems are characterized by a stochastic input process and an output mechanism that intermittently clears the system, i.e., instantaneously restores the net quantity in the system to zero \cite{Stid74}.}"  In the seminal paper by Stidham \cite{Stid74}, the system is cleared when the quantity in the system, exceeds a fixed threshold, and the explicit expression of the limiting distribution of the quantity in the system is derived. The concepts of stochastic clearing theory have found wide spread applications in the context of supply chain and service operations management. Historical application domains include bulk service queues, inventory control, and transportation planning (e.g., vehicle dispatching and shipment consolidation).

Applications to bulk-service queues and $(s, S)$ inventory systems are introduced in \cite{Stid77} where the goal is to compute the optimal level of $q$ to minimize the average cost including the fixed clearing and variable holding costs. Also, see \cite{EM13} for more recent work, addressing a stochastic clearing problem arising in the context of a queuing theoretic transportation application, along with related literature. In \cite{KPS03}, the stock level process is assumed to be a superposition of a drifted Brownian motion and a compound Poisson process, reflected at zero, and some cost functionals for this system are introduced under several clearing policies. 

In the context of vehicle dispatching, the problem is to determine the capacity and instants in time at which vehicles are dispatched leading to a dispatching (clearing) policy \cite{RT82,ROSS69,TZ79,ZT80}. In another seminal paper by Ross \cite{ROSS69}, the goal is to compute an optimal dispatching policy under a Poisson input process $N(t)$ with rate $\lambda(t)$, where all items are dispatched at time $T$. An intermediate dispatch time needs to be selected to  minimize the total waiting time of all items, and it is demonstrated that the optimal intermediate dispatch time should be the smallest $t$, such that $N(t)\geq\lambda(t)(T-t)$. The vehicle dispatching with non-stationary Poisson arrival is studied in \cite{RT82}, and the optimal dispatching policy is characterized by using an impulsive control approach. Three alternative policies for vehicle dispatching are discussed in \cite{TZ79}: (i) a $C$-capacity policy; (ii) a dispatching frequency policy $T$; (iii) a $(T,C)$ policy. The average cost models are derived under the three policies, and two firm models with cooperative and non-cooperative solution modes are discussed. In \cite{ZT80}, a random vehicle dispatching problem with options to send rented vehicles is considered, and it determines the firm's optimal fleet size to minimize the average cost.
    
Shipment consolidation is the strategy of combining small size shipments or customer orders, i.e., input process realizations, into a larger load with the goals of achieving scale economies and increasing resource, i.e., truck, utilization \cite{CETI04,CB03,CL00} at the time of shipment-release. In the context of shipment consolidation, customer orders represent the stochastic input process. The consolidated loads are released at some specific times that correspond to clearing instances, and the consolidation practice is subject to fixed clearing (shipment-release) and variable holding costs. The problem at hand is to make decisions regarding the timing and quantity of consolidated shipments, leading to a clearing policy. Hence, a shipment consolidation system can be treated as a stochastic clearing system. The consolidation practice is well-justified because  a fixed cost is incurred associated with each shipment-release, so it is not economical to dispatch the customer orders immediately, i.e., one by one. 

In this paper, motivated by a fundamental application in shipment consolidation, we revisit the notion of service performance for stochastic clearing system operation. All stochastic clearing systems are subject to service delay due to the inherent clearing practice. Our goal here is to evaluate and compare such delay under alternative operational policies for clearing decisions, as quantified by a basic measure of timely service (service performance), referred as \emph{Average Order Delay} ($AOD$). Although stochastic clearing theory has a long history, existing literature on the analysis of service performance, in general, and $AOD$, in particular, has several limitations as we discuss next in Section \ref{lit}. Hence, we extend the previous literature on stochastic clearing by proposing a more general method for a generic analytical derivation of $AOD$ for any renewal-type clearing policy, including but not limited to alternative shipment consolidation policies investigated earlier. 

The remainder of this paper is organized as follows. A detailed summary of previously established results are summarized in Section \ref{lit}. In Section \ref{sec:Average-Order-Delay},
using a martingale associated with the underlying stochastic input (Poisson order) process, we
provide a unified formula to calculate $AOD$ under different renewal-type
clearing policies of interest. Section \ref{sec:Properties-on-Truncated}
includes the properties of truncated random variables of interest, which are the key to prove the comparative results among alternative policies in
terms of $AOD$. The comparative results under a fixed shipment-release frequency are given in Section \ref{sec:Comparison-of-AOD}, and additional results
under fixed policy parameters are discussed in Section \ref{sec:Comparison-of-AOD-under fixed q and T}.
Finally, the paper concludes in Section \ref{sec:Conclusion}

\section{Relevant Literature \& Contributions of Current Paper}
\label{lit}
Within the existing literature on stochastic clearing, the paper that is most closely related to our work is \cite{CMW14} which is motivated by shipment consolidation applications. More specifically, in \cite{CMW14}, the customer orders (stochastic input) follow a Poisson process $N(t)$ with rate $\lambda$ and the problem at hand is to make decisions regarding the timing and quantity of consolidated shipments (clearing decisions). The paper places an emphasis on the importance of examining the timely delivery (i.e., service delay) implications of shipment consolidation under practically motivated policies with clearing characteristics. To the best of our knowledge, this is the only previous paper that pursues an \emph{analytical characterization} of service performance in the context of a clearing system, in general, and a shipment consolidation application, in particular. \emph{Simulation studies} addressing the notion of service delay (as it relates to inventory holding and customer waiting) in shipment consolidation can be found in \cite{HB94} and \cite{CML06}. For a detailed account of the analytical models for shipment consolidation, see \cite{CETI04}. For more recent analytical work emphasizing computation of optimal shipment consolidation policies and their applications in vendor-managed inventory systems and last mile distribution can be found in \cite{BCH11,CHB14,CMW14,CTL08,CML06,CL17,KO13,LWZ16,MCB10,MC10,SEB18,UB12}. 

The analysis in \cite{CMW14} is based on the general theory of renewal/Poisson processes, and it offers some basic results for computing and comparing various service measures of practical interest. However, the approach has limitations in the sense that the analysis does not result in a full-scale analytical comparison of $AOD$ among the policies. Instead, the comparative results are supported by a numerical investigation, and, hence, the conclusions are dependent on parametric settings considered. 

Here, we revisit the exact same problem setting in \cite{CMW14} by treating it as a stochastic clearing system. We are able extend the previous analysis by proposing a generic analytical approach amenable for a comparative analysis applicable under any renewal-type clearing policy, including but not limited to the alternative shipment consolidation policies identified and studied in the literature \cite{BCH11,CHB14,CMW14,MCB10}. 
Unlike the previous work in the area, our proposed method  utilizes a new martingale point of view and lends itself for a full analytical characterization of $AOD$, leading to a complete comparative analysis among renewal-type clearing policies. In particular, we are able to close the gaps in the previous literature via a complete set of analytically provable results (some of which were only demonstrated through numerical tests in \cite{CMW14}) along with a more general approach for examining service performance of stochastic clearing systems.

We consider the three general types of clearing policies common in shipment consolidation practice but applicable in the context of other stochastic clearing applications in queuing systems, inventory control, and transportation: (i) the first one is  quantity-based policy (QP) which achieves economies of scale; (ii) the second is called time-based policy (TP) which assures timely delivery; and (iii) the last one is the hybrid policy (HP) which is aimed at balancing the tradeoff between economies of scale and timely delivery. For the sake of completeness, we recall the definitions of the specific policies examined here and in the previous literature as described exactly in \cite{CMW14} where the term (i) "order" refers to "input/input process", (ii) "consolidated load/shipment" refers to "accumulated input", (iii) "release/shipment-release" refers to "clear/clearing":
\begin{itemize}
\item QP is aimed at consolidating a load of $q$ units before releasing
a shipment;
\item Under TP1, a shipment is released every $T$ units of time, and all orders that arrive between the two shipment-release epochs are consolidated;
\item Under TP2, the arrival time of the first order after a shipment-release decision is recorded, and the next shipment is released $T$ time units after the arrival time of the first order;
\item Under HP1, the goal is to consolidate a load of size $q$. However,
if the time since the last shipment-release epoch exceeds $T$ , then a shipment-release decision is made anyways;
\item Under HP2, the goal is also to consolidate a load of size $q$; but,
if the waiting time of the first order after the last shipment-release epoch exceeds $T$, then a consolidated load is released immediately.
\end{itemize}
Under TP1 and HP1, there may be empty shipments, which happens when
$N(T)=0$. In order to address this issue, we propose two revised policies, i.e., revised TP1 and revised
HP1, which do not allow empty shipments.
\begin{itemize}
\item Under revised TP1, a shipment is released every $T$ units of time as
long as the consolidated load is not 0. However, if there is no order
arriving within $T$ units of time since the last shipment, we do
not release a shipment, but continue consolidating for another multiple of $T$ units of time.
\item Under revised HP1, the goal is to consolidate a load of size $q$.
However, if the time since the last shipment epoch exceeds $T$ and
the consolidated load is positive, then the load is dispatched; on
the other hand, if the time since the last shipment exceeds $T$ and
the consolidated load is zero, we do not release a shipment and the system restarts.
\end{itemize}
Observe that we have three sets of policies: \{HP1, QP, TP1\}, \{HP2,
QP, TP2\}, and \{revised HP1, QP, revised TP1\}. Notice that within
each set, HPs would degenerate to QP when the time parameter goes
to infinity while degenerate to TPs when the quantity parameter goes
to infinity.

As we have noted earlier, all stochastic clearing systems are subject to service delay due to the inherent clearing practice. More specifically for the case of shipment consolidation practices, all of the above policies are implemented at the expense of customer waiting since a prolonged order holding is needed to accumulate a large load \cite{CETI04,CB03,CL00}. Hence, a key indicator of service performance is indeed $AOD$, which is the average waiting time of orders before delivery \cite{CMW14}. Under all of the   policies described above, the system is cleared once the consolidated load is released, and, hence, the consolidated load under each policy forms a regenerative process. In turn, the policies of interest are renewal-type clearing policies, and the $AOD$ can be obtained by applying the Renewal Reward Theorem, i.e.,
\[
AOD=\frac{\mathrm{\mathbb{E}}[\mbox{Cumulative waiting per consolidation cycle}]}{\mathrm{\mathbb{E}}[\mbox{Number of orders arriving in a consolidation cycle}]}=\frac{\mathrm{\mathbb{E}}[W]}{\lambda\mathrm{\mathbb{E}}[C]},
\]
where $W$ denotes the sum of the waiting times of the orders within
a consolidation cycle, and $C$ denotes the consolidation cycle length.
We index $AOD,W$, and $C$ by policy type as needed. Our specific objective here is to provide a comparative analysis of $AOD$ under the  alternative policies \{HP1, QP, TP1\}, \{HP2, QP, TP2\}, and \{revised HP1, QP, revised TP1\}. 

In \cite{CMW14}, the expectation of cumulative waiting per consolidation cycle under each policy is calculated individually, and the expressions of $AOD$ under
HP1 and HP2 are rather complex/involved limiting the ability to offer a complete analytical comparison of the policies at hand. Here, by characterizing each renewal-type clearing policy as a stopping time, we are able to provide a unified method to calculate $AOD$ from a martingale point of view with
the aid of optional stopping theorem. Further, we are able to derive the explicit expressions of $AOD$ under HP1 and HP2, which include terms involving first and second moments of truncated Poisson random variables whose refined properties, to the best of our knowledge, have not been examined previously. More specifically, throughout this work, for a non-negative integer-valued
random variable $X$, and a positive integer $q$, we denote
\[
X_{q}\triangleq\min(X,q)
\]
as the truncated random variable of interest. 
To provide the comparative results in terms of $AOD$ under alternative
policies, we develop refined properties of these truncated random variables. 
The resulting properties, in turn, are essential for the $AOD$ comparisons among
alternative policies. Specifically, we obtain the
following interesting results of the preservation property, which
have their own merits:
\begin{itemize}
\item Given a non-negative integer-valued random variable $Y$ with $\mathbb{VAR}[Y]\leq\mathrm{\mathbb{E}}[Y]<\infty$,
for any positive integer $N$, we have $\mathbb{VAR}[Y_{N}]<\mathrm{\mathbb{E}}[Y_{N}]$ (see Lemma \ref{VAR<E}).
This, in turn, implies that the relationship of $\mathbb{VAR}[Y]\leq\mathrm{\mathbb{E}}[Y]$
is preserved after $Y$ is truncated.
\item Let $X,Y$ be two non-negative integer valued random variables, and
$X$ is stochastically larger than $Y$. If $\mathrm{\mathbb{E}}[X_{q}]\leq\mathrm{\mathbb{E}}[Y_{q+1}]$,
where $q$ is a positive integer, then $\mathrm{\mathbb{E}}[X_{q}^{2}]\leq\mathrm{\mathbb{E}}[Y_{q+1}^{2}]$ (see Lemma \ref{E2<E2}).
This property implies that the relationship between first moments of two random
variables is preserved in second moment setting.
\item Suppose $X\sim\text{Poisson}(\lambda)$ and $N$ is a positive integer,
then $\mathrm{\mathbb{E}}[X_{N}^{2}]/\mathrm{\mathbb{E}}[X_{N}]$
is increasing with respect to $\lambda$ (see Lemma \ref{ratiobetween2and1moment}). 
We know $\mathrm{\mathbb{E}}[X^{2}]/\mathrm{\mathbb{E}}[X]=\lambda+1$,
which is increasing with respect to $\lambda$. That is, after $X$ is truncated, the property is still preserved.
\end{itemize}

Based on the above refined properties, we provide the following comparative, insightful results among alternative policies in terms of $AOD$: For a fixed expected consolidation cycle length:
\begin{itemize}
\item QP outperforms all renewal-type clearing policies in terms of
$AOD$ (see Theorem \ref{Qbest}).
\item The general class of HPs performs better than the general class of
counterpart TPs in terms of $AOD$ (see Theorems \ref{HPbetterthanTP} and \ref{revised}. Three sets: \{HP1, TP1\}, \{HP2,
TP2\}, and \{revised HP1, revised TP1\}).
\item The HP1 with larger quantity parameter would achieve larger $AOD$ than
the HP1 with smaller quantity parameter, and the same property holds
for HP2 and revised HP1 (see Theorems \ref{HPvsHPundersameEC} and \ref{revised}).
\end{itemize}
Furthermore, with fixed parameters, we show that
\begin{itemize}
\item The general class of HPs has less $AOD$ than the general classes of
counterpart QPs and TPs (see Theorems \ref{HP1bestfixed} and \ref{HP2bestfixed}. Three sets: \{HP1, QP, TP1\}, \{HP2, QP,
TP2\}, and \{revised HP1, QP, revised TP1\}).
\item HP1 has the same $AOD$ as revised HP1, and less $AOD$ than HP2 (see Theorem \ref{HP1betterthanHP2fixed} and Appendix \ref{sec:Append}).
\end{itemize}

Next, we proceed with the details of the new analytical approach for examining $AOD$ which has provided the foundation of the insightful results summarized above. 

\section{Average Order Delay ($AOD$)\label{sec:Average-Order-Delay}}

In this section, based on a martingale associated with the Poisson
input process, we provide a unified method to calculate the $AOD$
for any renewal-type clearing policy. The following lemma reveals
this martingale.
\begin{lem}
\label{martingale} Let $N(t)$ is a Poisson process with rate $\lambda$,
and define $W(t)\triangleq\int_{0}^{t}N(u)du$. Then,
\[
\left\{ W(t)-\frac{1}{2\lambda}N^{2}(t)+\frac{1}{2\lambda}N(t)\right\} _{t\geq0}
\]
is a martingale with respect to the natural filtration $\{\mathcal{G}_{t}\}$, which
is the $\sigma$ field generated by the family of random variables
$\left\{ N(s),s\in[0,t]\right\} $.
\end{lem}

\begin{proof}
Since the Poisson process is with stationary independent increment,
for $s<t$, we have,
\begin{eqnarray*}
\mathrm{\mathbb{E}}\Bigl[\int_{0}^{t}N(u)du\mid\mathcal{G}_{s}\Bigr] & = & \int_{0}^{s}N(u)du+\mathrm{\mathbb{E}}\Bigl[\int_{s}^{t}N(u)du\mid\mathcal{G}_{s}\Bigr]\\
 & = & \int_{0}^{s}N(u)du+(t-s)N(s)+\mathrm{\mathbb{E}}\Bigl[\int_{0}^{t-s}N(u)du\Bigr]\\
 & = & \int_{0}^{s}N(u)du+(t-s)N(s)+\frac{1}{2}\lambda(t-s)^{2},
\end{eqnarray*}
\[
\frac{1}{2\lambda}\mathrm{\mathbb{E}}[N^{2}(t)\mid\mathcal{G}_{s}]=\frac{1}{2\lambda}\Bigl(N^{2}(s)+2\lambda(t-s)N(s)+\lambda(t-s)+\lambda^{2}(t-s)^{2}\Bigr),
\]
and
\[
\frac{1}{2\lambda}\mathrm{\mathbb{E}}[N(t)\mid\mathcal{G}_{s}]=\frac{1}{2\lambda}(N(s)+\lambda(t-s)).
\]
Therefore,
\[
\mathrm{\mathbb{E}}\Bigl[\int_{0}^{t}N(u)du-\frac{1}{2\lambda}N^{2}(t)+\frac{1}{2\lambda}N(t)\mid\mathcal{G}_{s}\Bigr]=\int_{0}^{s}N(u)du-\frac{1}{2\lambda}N^{2}(s)+\frac{1}{2\lambda}N(s),
\]
which shows that $W(t)-\frac{1}{2\lambda}N^{2}(t)+\frac{1}{2\lambda}N(t)$
is a martingale.
\end{proof}
The shipment-release time is always a stopping time with respect to the
stochastic customer order process. For a consolidation policy with shipment-release
time $\tau$, the cumulative waiting time within one consolidation
cycle is
\[
W(\tau)=\int_{0}^{\tau}N(u)du.
\]

Throughout this work, for two real numbers $a$ and $b$, denote $a\wedge b\triangleq\min(a,b)$.
From Lemma \ref{martingale} and the martingale stopping theorem \cite[Theorem 6.2.2]{Ross96},
we have that for any stopping time $\tau$ and any fixed $t>0$,
\begin{eqnarray}
\mathrm{\mathbb{E}}[W(\tau\wedge t)]=\frac{1}{2\lambda}\mathrm{\mathbb{E}}[N^{2}(\tau\wedge t)-N(\tau\wedge t)].\label{MST_W}
\end{eqnarray}
Also, apply again the martingale stopping theorem on another martingale
$\left\{ N(t)-\lambda t\right\} _{t\geq0}$, we have that for any
stopping time $\tau$ and any fixed $t>0$,
\begin{eqnarray}
\mathrm{\mathbb{E}}[N(\tau\wedge t)]=\lambda\mathrm{\mathbb{E}}[\tau\wedge t].\label{MST_N}
\end{eqnarray}
We consider a renewal-type clearing policy with finite cycle mean,
i.e., $\tau$ is of finite mean. From the monotone convergence theorem
\cite[Theorem 2.3.4]{athreya2006measure},
\begin{eqnarray*}
\lim_{t\rightarrow\infty}\mathrm{\mathbb{E}}[W(\tau\wedge t)]=\mathrm{\mathbb{E}}[W(\tau)],
\end{eqnarray*}
and
\begin{eqnarray*}
\lim_{t\rightarrow\infty}\mathrm{\mathbb{E}}[N(\tau\wedge t)(N(\tau\wedge t)-1)]=\mathrm{\mathbb{E}}[N(\tau)(N(\tau)-1)].
\end{eqnarray*}
Noticing Eq. (\ref{MST_W}), we obtain that the expectation of cumulative
waiting time within one consolidation cycle is
\begin{eqnarray}
\mathrm{\mathbb{E}}[W(\tau)]=\frac{1}{2\lambda}\mathrm{\mathbb{E}}[N^{2}(\tau)-N(\tau)].\label{W}
\end{eqnarray}
Similarly, from Eq. (\ref{MST_N}), we have the expectation of consolidation
cycle is
\begin{eqnarray}
\mathrm{\mathbb{E}}[\tau]=\frac{1}{\lambda}\mathrm{\mathbb{E}}[N(\tau)].\label{C}
\end{eqnarray}

From the above discussion, we can deduce $AOD$ for any renewal-type
shipment consolidation policy. Now we calculate the $AOD$ under different
shipment consolidation policies of interest here as follows. For two random
variables $X$ and $Y$, we use $X\sim Y$ to denote that they follow
the same distribution.

1. QP with integer parameter $q\geq1$: $\tau=\tau_{q}$, the time
until the $q$-th order, $q$ is a positive integer; $N(\tau_{q})=q$.
So,
\begin{eqnarray*}
 &  & \mathrm{\mathbb{E}}[W_{QP}]=\mathrm{\mathbb{E}}[W(\tau_{q})]=\frac{1}{2\lambda}q(q-1),\\
 &  & \mathrm{\mathbb{E}}[C_{QP}]=\mathrm{\mathbb{E}}[\tau_{q}]=\frac{q}{\lambda}.
\end{eqnarray*}

2. TP1 with parameter $T$: $\tau=T$, a constant; $N(T)\sim\text{Poisson}(\lambda T)$.
So
\begin{eqnarray*}
 &  & \mathrm{\mathbb{E}}[W_{TP1}]=\mathrm{\mathbb{E}}[W(T)]=\frac{1}{2\lambda}\mathrm{\mathbb{E}}[N^{2}(T)-N(T)]=\frac{1}{2}\lambda T^{2},\\
 &  & \mathrm{\mathbb{E}}[C_{TP1}]=T.
\end{eqnarray*}

3. TP2 with parameter $T$: $\tau=\tau_{1}+T$; $N(\tau_{1}+T)\sim1+N(T)$.
So
\begin{eqnarray*}
 &  & \mathrm{\mathbb{E}}[W_{TP2}]=\mathrm{\mathbb{E}}[W(\tau_{1}+T)]=\frac{1}{2\lambda}\mathrm{\mathbb{E}}[N^{2}(T)+N(T)]=\frac{1}{2}\lambda T^{2}+T,\\
 &  & \mathrm{\mathbb{E}}[C_{TP2}]=\frac{1}{\lambda}+T.
\end{eqnarray*}

4. HP1 with parameters $q$ and $T$: $\tau=\tau_{q}\wedge T$; Let
$Y\sim\text{Poisson}(\lambda T)$, $N\left(\tau_{q}\wedge T\right)\sim N(T)\wedge q\sim Y_{q}$,
which is a truncated Poisson random variable. So
\begin{eqnarray*}
 &  & \mathrm{\mathbb{E}}[W_{HP1}]=\mathrm{\mathbb{E}}\left[W\left(\tau_{q}\wedge T\right)\right]=\frac{1}{2\lambda}\mathrm{\mathbb{E}}[Y_{q}(Y_{q}-1)],\\
 &  & \mathrm{\mathbb{E}}[C_{HP1}]=\frac{1}{\lambda}\mathrm{\mathbb{E}}[Y_{q}].
\end{eqnarray*}

5. HP2 with parameters $q$ and $T$: $\tau=\tau_{q}\wedge(\tau_{1}+T)$;
Let $Y\sim\text{Poisson}(\lambda T)$, $N\left(\tau_{q}\wedge(\tau_{1}+T)\right)\sim(1+N(T))\wedge q\sim Y_{q-1}+1$.
So,
\begin{eqnarray*}
 &  & \mathrm{\mathbb{E}}[W_{HP2}]=\mathrm{\mathbb{E}}\left[W\left(\tau_{q}\wedge(\tau_{1}+T)\right)\right]=\frac{1}{2\lambda}\mathrm{\mathbb{E}}[Y_{q-1}(Y_{q-1}+1)],\\
 &  & \mathrm{\mathbb{E}}[C_{HP2}]=\frac{1}{\lambda}\mathrm{\mathbb{E}}[Y_{q-1}+1].
\end{eqnarray*}

In Table \ref{t1}, we summarize the $AOD$ for different shipment consolidation
policies. We would notice that the expressions of $AOD_{HP1}$ and
$AOD_{HP2}$ involve the truncated Poisson random variables, which
are much simplified than the expressions in \cite{CMW14}. Note that
under TP1 and HP1, the consolidation cycle clock starts over, even
if no order arrives within the previous cycle. We consider the correspondingly
revised policies in the Appendix, which do not allow empty dispatches.

\begin{table}[h]
\centering{}\hspace*{-0.5in} %
\begin{tabular}{llll}
\hline
\vspace{-0.1in}
 &  &  & \tabularnewline
$AOD_{\tau}$  & $=$  & {\Large{}$\;\frac{\mathrm{\mathbb{E}}\left[W(\tau)\right]}{\lambda\mathrm{\mathbb{E}}\left[C_{\tau}\right]}=\frac{\mathrm{\mathbb{E}}[N^{2}(\tau)-N(\tau)]/(2\lambda)}{\mathrm{\mathbb{E}}[N(\tau)]}$}  & \vspace{0.05in}
\tabularnewline
$AOD_{QP}$  & $=$  & {\Large{}$\;\frac{\mathrm{\mathbb{E}}\left[W_{QP}\right]}{\lambda\mathrm{\mathbb{E}}\left[C_{QP}\right]}=\frac{(q-1)q/2\lambda}{q}=\frac{q-1}{2\lambda}$}  & \vspace{0.05in}
\tabularnewline
$AOD_{TP1}$  & $=$  & {\Large{}$\;\frac{\mathrm{\mathbb{E}}\left[W_{TP1}\right]}{\lambda\mathrm{\mathbb{E}}\left[C_{TP1}\right]}=\frac{\lambda T^{2}/2}{\lambda T}=\frac{1}{2}T$}  & \vspace{0.05in}
\tabularnewline
$AOD_{TP2}$  & $=$  & {\Large{}$\;\frac{\mathrm{\mathbb{E}}\left[W_{TP2}\right]}{\lambda\mathrm{\mathbb{E}}\left[C_{TP2}\right]}=\frac{T+\lambda T^{2}/2}{1+\lambda T}$}  & \vspace{0.05in}
\tabularnewline
$AOD_{HP1}$  & $=$  & {\Large{}$\;\frac{\mathrm{\mathbb{E}}\left[W_{HP1}\right]}{\lambda\mathrm{\mathbb{E}}\left[C_{HP1}\right]}=\frac{\mathrm{\mathbb{E}}[Y_{q}(Y_{q}-1)]/(2\lambda)}{\mathrm{\mathbb{E}}[Y_{q}]}$}  & \vspace{0.1in}
\tabularnewline
$AOD_{HP2}$  & $=$  & {\Large{}$\;\frac{\mathrm{\mathbb{E}}\left[W_{HP2}\right]}{\lambda\mathrm{\mathbb{E}}\left[C_{HP2}\right]}=\frac{\mathrm{\mathbb{E}}[Y_{q-1}(Y_{q-1}+1)]/(2\lambda)}{\mathrm{\mathbb{E}}[Y_{q-1}+1]}$}  & \vspace{0.1in}
\tabularnewline
\hline
\end{tabular}\caption{Summary of the expressions of $AOD$, where $Y\sim\text{Poisson}(\lambda T)$.}
\label{t1}
\end{table}

\section{Properties of Truncated Random Variables\label{sec:Properties-on-Truncated}}

In this section, we investigate the properties of truncated random
variables, which are connected to the comparison of different shipment consolidation policies of interest here in terms of $AOD$.
\begin{lem}
\label{VAR-VAR}Given an integer valued random variable $Y$, and
a positive integer $M$, we have
\[
\mathbb{VAR}[Y]-\mathbb{VAR}[Y_{M}]=\mathbb{VAR}[Y-Y_{M}]+2(M-\mathrm{\mathbb{E}}[Y_{M}])(\mathrm{\mathbb{E}}[Y]-\mathrm{\mathbb{E}}[Y_{M}]).
\]
In particular, $\mathbb{VAR}[Y_{M}]<\mathbb{VAR}[Y]$.
\end{lem}

\begin{proof}
First, notice

\[
Y-Y_{M}=(Y-M)1_{Y\geq M+1},
\]
and

\[
\mathrm{\mathbb{E}}[Y_{M}(Y-Y_{M})]=M\mathrm{\mathbb{E}}[(Y-M)1_{Y\geq M+1}]=M(\mathrm{\mathbb{E}}[Y]-\mathrm{\mathbb{E}}[Y_{M}]).
\]
Therefore,
\begin{eqnarray*}
\mathbb{COV}(Y_{M},Y-Y_{M})= &  & \mathrm{\mathbb{E}}[Y_{M}(Y-Y_{M})]-\mathrm{\mathbb{E}}[Y_{M}]\mathrm{\mathbb{E}}[Y-Y_{M}]\\
= &  & (M-\mathrm{\mathbb{E}}[Y_{M}])(\mathrm{\mathbb{E}}[Y]-\mathrm{\mathbb{E}}[Y_{M}]).
\end{eqnarray*}
We have
\begin{eqnarray*}
\mathbb{VAR}[Y]= &  & \mathbb{VAR}[Y_{M}+(Y-Y_{M})]\\
= &  & \mathbb{VAR}[Y_{M}]+\mathbb{VAR}[Y-Y_{M}]+2\mathbb{COV}(Y_{M},Y-Y_{M})\\
= &  & \mathbb{VAR}[Y_{M}]+\mathbb{VAR}[Y-Y_{M}]+2(M-\mathrm{\mathbb{E}}[Y_{M}])(\mathrm{\mathbb{E}}[Y]-\mathrm{\mathbb{E}}[Y_{M}]).
\end{eqnarray*}
\end{proof}
Based on Lemma \ref{VAR-VAR}, we establish the following result,
which will be useful in the comparison between the general class of
HPs and the general class of counterpart TPs in terms of $AOD$.
\begin{lem}
\label{VAR<E}Given an nonnegative integer-valued random variable
$Y$ with $\mathbb{VAR}[Y]\leq\mathrm{\mathbb{E}}[Y]<\infty$, then,
for any positive integer $N$, we have $\mathbb{VAR}[Y_{N}]<\mathrm{\mathbb{E}}[Y_{N}]$.
In particular, $\mathbb{VAR}[Y_{N}]<\mathrm{\mathbb{E}}[Y_{N}]$,
if $Y$ is a Poisson random variable.
\end{lem}

\begin{proof}
Define $f(N)\triangleq\mathbb{VAR}[Y_{N}]-\mathrm{\mathbb{E}}[Y_{N}]$,
for $N\geq1$. Noticing
\[
Y_{N}=Y_{N+1}\wedge N,\,Y_{N+1}-Y_{N}=1_{Y\geq N+1},
\]
and applying Lemma \ref{VAR-VAR}, we have
\begin{eqnarray*}
 &  & f(N+1)-f(N)\\
= &  & (\mathbb{VAR}[Y_{N+1}]-\mathbb{VAR}[Y_{N}])-(\mathrm{\mathbb{E}}[Y_{N+1}]-\mathrm{\mathbb{E}}[Y_{N}])\\
= &  & \mathbb{VAR}[Y_{N+1}-Y_{N}]+2(N-\mathrm{\mathbb{E}}[Y_{N}])(\mathrm{\mathbb{E}}[Y_{N+1}]-\mathrm{\mathbb{E}}[Y_{N}])-(\mathrm{\mathbb{E}}[Y_{N+1}]-\mathrm{\mathbb{E}}[Y_{N}])\\
= &  & \mathbb{VAR}[Y_{N+1}-Y_{N}]+(2N-2\mathrm{\mathbb{E}}[Y_{N}]-1)(\mathrm{\mathbb{E}}[Y_{N+1}]-\mathrm{\mathbb{E}}[Y_{N}])\\
= &  & \mathbb{P}(Y\geq N+1)\mathbb{P}(Y\leq N)+(2N-2\mathrm{\mathbb{E}}[Y_{N}]-1)\mathbb{P}(Y\geq N+1)\\
= &  & \left(2\mathrm{\mathbb{E}}[\max(N-Y,0)]-\mathbb{P}(Y\geq N+1)\right)\mathbb{P}(Y\geq N+1).
\end{eqnarray*}
Obviously, $2\mathrm{\mathbb{E}}[\max(N-Y,0)]-\mathbb{P}(Y\geq N+1)$
is increasing with respect to $N$, which implies that $f(N+1)-f(N)$
changes sign at most once with respect to $N$: either from negative
to positive or always positive. In the first case, $f(N)$ is first
decreasing and then increasing; in the second case, $f(N)$ is increasing.
Based on the following observations

\[
\lim_{N\rightarrow\infty}f(N)=\lim_{N\rightarrow\infty}(\mathbb{VAR}[Y_{N}]-\mathrm{\mathbb{E}}[Y_{N}])=\mathbb{VAR}[Y]-\mathrm{\mathbb{E}}[Y]\leq0,
\]
and
\[
f(1)=\mathbb{VAR}[1_{Y\geq1}]-\mathrm{\mathbb{E}}[1_{Y\geq1}]=\mathbb{P}(Y\geq1)(1-\mathbb{P}(Y\geq1))-\mathbb{P}(Y\geq1)<0,
\]
we conclude that $f(N)<0$ for all $N$, i.e. $\mathbb{VAR}[Y_{N}]<\mathrm{\mathbb{E}}[Y_{N}]$.
\end{proof}
Next, we provide a result which would be essential in comparing the
same type HP with different parameters in terms of $AOD$, with a fixed
expected consolidation cycle length $\mathrm{\mathbb{E}}[C]$.
\begin{lem}
\label{E2<E2}Assume $X,Y$ are two integer valued random variables,
and $X$ is stochastically larger than $Y$. If $\mathrm{\mathbb{E}}[X_{q}]\leq\mathrm{\mathbb{E}}[Y_{q+1}]$,
where $q$ is a positive integer, then $\mathrm{\mathbb{E}}[X_{q}^{2}]\leq\mathrm{\mathbb{E}}[Y_{q+1}^{2}]$.
\end{lem}

\begin{proof}
From
\[
Y_{q+1}^{2}-Y_{q}^{2}=(2q+1)(Y_{q+1}-Y_{q}),
\]
we have
\begin{eqnarray*}
\mathrm{\mathbb{E}}[Y_{q+1}^{2}]-\mathrm{\mathbb{E}}[Y_{q}^{2}]=(2q+1)(\mathrm{\mathbb{E}}[Y_{q+1}]-\mathrm{\mathbb{E}}[Y_{q}])\geq(2q+1)(\mathrm{\mathbb{E}}[X_{q}]-\mathrm{\mathbb{E}}[Y_{q}]).
\end{eqnarray*}
Therefore,
\begin{eqnarray}
\mathrm{\mathbb{E}}[Y_{q+1}^{2}]-\mathrm{\mathbb{E}}[X_{q}^{2}]\geq &  & \mathrm{\mathbb{E}}[Y_{q}^{2}]-\mathrm{\mathbb{E}}[X_{q}^{2}]+(2q+1)(\mathrm{\mathbb{E}}[X_{q}]-\mathrm{\mathbb{E}}[Y_{q}])\nonumber\\
= &  & \mathrm{\mathbb{E}}[(X_{q}-Y_{q})(2q+1-X_{q}-Y_{q})].\label{OK}
\end{eqnarray}
Since $X$ is stochastically
larger than $Y$, $X_{q}$ is also stochastically larger than $Y_{q}$.
From \cite[Proposition 9.2.2]{Ross96}, we always can find two random
variables $X'$ and $Y'$, such that $X'$ has the same probability
distribution as $X_{q}$, $Y'$ has the same probability distribution
as $Y_{q}$, and $X'\geq Y'$ almost surely.

From Eq. (\ref{OK}), and notice $X'\leq q$, $Y'\leq q$ almost surely,
we have
\begin{eqnarray*}
\mathrm{\mathbb{E}}[Y_{q+1}^{2}]-\mathrm{\mathbb{E}}[X_{q}^{2}]\geq\mathrm{\mathbb{E}}[(X'-Y')(2q+1-X'-Y')]\geq0.
\end{eqnarray*}
\end{proof}
The following lemma characterizes how the ratio between the second
moment and the first moment of a truncated Poisson random variable
changes with respect to the Poisson rate parameter, which will be
used when we compare HP1 and HP2 under fixed policy parameters, in
terms of $AOD$.
\begin{lem}
\label{ratiobetween2and1moment}Suppose $X\sim\text{Poisson}(\lambda)$
and $N$ is a positive integer, then $\mathrm{\mathbb{E}}[X_{N}^{2}]/\mathrm{\mathbb{E}}[X_{N}]$
is increasing with respect to $\lambda$.
\end{lem}

\begin{proof}
Let $Y\sim\text{Poisson}(\lambda_{1})$, $Z\sim\text{Poisson}(\lambda_{2})$,
where $\lambda_{1}<\lambda_{2}$. When $k<m<N$,
\begin{eqnarray}
\mathbb{P}(Z_{N}=m)\mathbb{P}(Y_{N}=k)-\mathbb{P}(Y_{N}=m)\mathbb{P}(Z_{N}=k)=\frac{e^{-\lambda_{1}-\lambda_{2}}}{m!k!}(\lambda_{2}^{m}\lambda_{1}^{k}-\lambda_{1}^{m}\lambda_{2}^{k})>0,\label{mk}
\end{eqnarray}
and when $k<N$,
\begin{eqnarray}
 &  & \mathbb{P}(Z_{N}=N)\mathbb{P}(Y_{N}=k)-\mathbb{P}(Y_{N}=N)\mathbb{P}(Z_{N}=k)\nonumber \\
= &  & \sum_{j\geq N}(\mathbb{P}(Z=j)\mathbb{P}(Y=k)-\mathbb{P}(Y=j)\mathbb{P}(Z=k))\nonumber \\
= &  & \sum_{j\geq N}\frac{e^{-\lambda_{1}-\lambda_{2}}}{j!k!}(\lambda_{2}^{j}\lambda_{1}^{k}-\lambda_{1}^{j}\lambda_{2}^{k})>0.\label{Nk}
\end{eqnarray}
Note that for any non-negative integer valued random variable $W$,
we have
\begin{eqnarray*}
\mathrm{\mathbb{E}}[W^{2}]=\sum_{m=1}^{\infty}m^{2}\mathbb{P}(W=m)=\sum_{m=1}^{\infty}\sum_{j=1}^{m}m\mathbb{P}(W=m)=\sum_{j=1}^{\infty}\sum_{m=j}^{\infty}m\mathbb{P}(W=m).
\end{eqnarray*}
Therefore, we obtain
\begin{eqnarray*}
 &  & \mathrm{\mathbb{E}}[Z_{N}^{2}]\mathrm{\mathbb{E}}[Y_{N}]-\mathrm{\mathbb{E}}[Y_{N}^{2}]\mathrm{\mathbb{E}}[Z_{N}]\\
= &  & \sum_{j=1}^{N}\sum_{m=j}^{N}m\mathbb{P}(Z_{N}=m)\sum_{k=1}^{N}k\mathbb{P}(Y_{N}=k)-\sum_{j=1}^{N}\sum_{m=j}^{N}m\mathbb{P}(Y_{N}=m)\sum_{k=1}^{N}k\mathbb{P}(Z_{N}=k)\\
= &  & \sum_{j=1}^{N}\sum_{m=j}^{N}\sum_{k=1}^{j-1}mk[\mathbb{P}(Z_{N}=m)\mathbb{P}(Y_{N}=k)-\mathbb{P}(Y_{N}=m)\mathbb{P}(Z_{N}=k)]\\
> &  & 0,
\end{eqnarray*}
where the second equality is derived from
\begin{eqnarray*}
\sum_{m=j}^{N}\sum_{k=j}^{N}mk[\mathbb{P}(Z_{N}=m)\mathbb{P}(Y_{N}=k)-\mathbb{P}(Y_{N}=m)\mathbb{P}(Z_{N}=k)]=0,
\end{eqnarray*}
and the last inequality holds by Eqs. (\ref{mk}) and (\ref{Nk}).

Therefore,
\begin{eqnarray*}
\frac{\mathrm{\mathbb{E}}[Z_{N}^{2}]}{\mathrm{\mathbb{E}}[Z_{N}]}-\frac{\mathrm{\mathbb{E}}[Y_{N}^{2}]}{\mathrm{\mathbb{E}}[Y_{N}]}=\frac{\mathrm{\mathbb{E}}[Z_{N}^{2}]\mathrm{\mathbb{E}}[Y_{N}]-\mathrm{\mathbb{E}}[Y_{N}^{2}]\mathrm{\mathbb{E}}[Z_{N}]}{\mathrm{\mathbb{E}}[Y_{N}]\mathrm{\mathbb{E}}[Z_{N}]}>0,
\end{eqnarray*}
which implies that $\frac{\mathrm{\mathbb{E}}[X_{N}^{2}]}{\mathrm{\mathbb{E}}[X_{N}]}$
is increasing with respect to $\lambda$.
\end{proof}

\section{Comparison of $AOD$ under a Fixed Expected Cycle Length, $\mathrm{\mathbb{E}}[C]$\label{sec:Comparison-of-AOD}}

In \textbf{(O10)} of \cite{CMW14}, through a numerical study, there
is an observation that for a given $\mathrm{\mathbb{E}}[C]$, the
QP performs the best and TPs perform the worst in terms of $AOD$. In
this section, we analytically show that for a given $\mathrm{\mathbb{E}}[C]$,
QP provides superior service compared with any other shipment consolidation
policy in terms of $AOD$, not limited to HPs and TPs. Further, we rigorously
compare HPs and TPs in terms of $AOD$, for a given $\mathrm{\mathbb{E}}[C]$.
In addition, for a given $\mathrm{\mathbb{E}}[C]$, we provide the
comparative result between the same type HP policies with different
parameters, in terms of $AOD$. The readers who are interested in the
managerial motivation of this comparison of $AOD$ are referred to \cite{CMW14}.
\begin{thm}
\label{Qbest}For a fixed expected consolidation cycle length, QP
outperforms all the other renewal-type clearing policies in terms
of $AOD$.
\end{thm}
\begin{proof}
From Table \ref{t1}, we know $AOD$ of a shipment consolidation policy
with shipment-release time $\tau$ is
\[
AOD_{\tau}=\frac{\mathrm{\mathbb{E}}[N^{2}(\tau)-N(\tau)]/(2\lambda)}{\mathrm{\mathbb{E}}[N(\tau)]}.
\]
From Eq. (\ref{C}), the fixed $\mathrm{\mathbb{E}}[\tau]$ implies $\mathrm{\mathbb{E}}[N(\tau)]$
is fixed. Furthermore, we have
\begin{eqnarray*}
AOD_{\tau}=\frac{1}{2\lambda}\left(\frac{\mathrm{\mathbb{E}}[N^{2}(\tau)]}{\mathrm{\mathbb{E}}[N(\tau)]}-1\right)\geq\frac{1}{2\lambda}(\mathrm{\mathbb{E}}[N(\tau)]-1),
\end{eqnarray*}
where the equality holds if and only if $N(\tau)$ is a constant,
which implies QP achieves the least $AOD$ with a fixed expected consolidation
cycle length.
\end{proof}
\begin{rem}
\label{remarkforcost}If there is a consolidation policy with shipment-release 
time $\tau$, which has the same expected cycle length as a quantity-based
policy with parameter $q$, that is $\mathrm{\mathbb{E}}[\tau]=\frac{q}{\lambda}$,
the average cost associated with this policy is
\[
\frac{A_{D}+C_{D}[N(\tau)]+\omega\mathrm{\mathbb{E}}[W(\tau)]}{\mathrm{\mathbb{E}}[\tau]},
\]
where $A_{D}$ is the fixed cost for each shipment-release, $C_{D}$ is the
unit transportation cost, and $\omega$ is the waiting cost per unit
per unit time. With fixed $\mathrm{\mathbb{E}}[\tau]$, $\mathrm{\mathbb{E}}[N(\tau)]$
is also fixed. From Theorem \ref{Qbest}, we conclude that the corresponding
quantity-based policy achieves less average cost than this policy
with shipment-release time $\tau$.
\end{rem}

One disadvantage of QP is that it has no upper bound on the cycle
length, in contrast, HP is of more practical importance since by definition
it has an upper bound on the cycle length. In the next result, we show the
general class of HPs outperforms the general class of counterpart
TPs in terms of $AOD$, under a fixed expected consolidation cycle length.
\begin{thm}
\label{HPbetterthanTP}For a fixed expected consolidation cycle length
$\mathrm{\mathbb{E}}[C]$, HP1 performs better than TP1, and HP2 performs
better than TP2 in terms of $AOD$.
\end{thm}

\begin{proof}
We consider a fixed $\mathrm{\mathbb{E}}[C]$ and use the following
notation for the corresponding policy parameters under this $\mathrm{\mathbb{E}}[C]$
value: TP1 with parameter $T_{1}$, TP2 with parameter $T_{2}$, HP1
with parameters $q_{H1}$ and $T_{H1}$, and HP2 with parameters $q_{H2}$
and $T_{H2}$. Recalling the $\mathrm{\mathbb{E}}[C]$ expressions
in Table \ref{t1}, we note that, by assumption,
\begin{equation}
\frac{1}{\lambda}\mathrm{\mathbb{E}}[X_{q_{H1}}]=T_{1},\label{HP1-TBP1}
\end{equation}
and
\begin{eqnarray}
\frac{1}{\lambda}\mathrm{\mathbb{E}}[1+Z_{q_{H2}-1}]=\frac{1}{\lambda}+T_{2},\label{HP2-TBP2}
\end{eqnarray}
where $X\sim\text{Poisson}(\lambda T_{H1})$, $Z\sim\text{Poisson}(\lambda T_{H2})$.
Next, recalling the results in Table \ref{t1} and the assumption
of fixed $\mathrm{\mathbb{E}}[C]$ values for all the policies of
interest, we need to show that

\begin{equation}
\mathrm{\mathbb{E}}[X_{q_{H1}}(X_{q_{H1}}-1)]<\lambda^{2}T_{1}^{2},\label{HP1 vs TBP1}
\end{equation}
and
\begin{eqnarray}
\mathrm{\mathbb{E}}[Z_{q_{H2}-1}(Z_{q_{H2}-1}+1)]<2\lambda T_{2}+\lambda^{2}T_{2}^{2}.\label{HP2 vs TBP2}
\end{eqnarray}
In fact, by recalling Eqs. (\ref{HP1-TBP1}) and (\ref{HP2-TBP2}), we
have
\begin{eqnarray*}
\mathrm{\mathbb{E}}[X_{q_{H1}}(X_{q_{H1}}-1)]=\mathbb{VAR}[X_{q_{H1}}]+\mathrm{\mathbb{E}}^{2}[X_{q_{H1}}]-\mathrm{\mathbb{E}}[X_{q_{H1}}]<\mathrm{\mathbb{E}}^{2}[X_{q_{H1}}]=\lambda^{2}T_{1}^{2},
\end{eqnarray*}
and
\begin{eqnarray*}
\mathrm{\mathbb{E}}[Z_{q_{H2}-1}(Z_{q_{H2}-1}+1)] &  & =\mathbb{VAR}[Z_{q_{H2}-1}]+\mathrm{\mathbb{E}}^{2}[Z_{q_{H2}-1}]+\mathrm{\mathbb{E}}[Z_{q_{H2}-1}]\\
 &  & <2\mathrm{\mathbb{E}}[Z_{q_{H2}-1}]+\mathrm{\mathbb{E}}^{2}[Z_{q_{H2}-1}]\\
 &  & =2\lambda T_{2}+\lambda^{2}T_{2}^{2},
\end{eqnarray*}
where the inequalities are derived from Lemma \ref{VAR<E}.
\end{proof}
\begin{rem}
Following the same argument as in Remark \ref{remarkforcost}, in
terms of average cost criterion, the general class of HPs also outperforms
the general class of counterpart TPs.
\end{rem}

From Lemma \ref{E2<E2}, we deduce a stronger result as follows, which
allows us to compare two HP policies of the same type under a fixed
$\mathrm{\mathbb{E}}[C]$.
\begin{thm}
\label{HPvsHPundersameEC}For a fixed expected consolidation cycle
length $\mathrm{\mathbb{E}}[C]$, the HP1 with larger quantity parameter
achieves larger $AOD$ than the HP1 with smaller quantity parameter,
and the similar result holds for HP2.
\end{thm}

\begin{proof}
We consider a fixed $\mathrm{\mathbb{E}}[C]$ and use the following
notation for the corresponding policy parameters under this $\mathrm{\mathbb{E}}[C]$
value: the first HP1 with parameters $q_{H}$ and $T_{H}$, the second
HP1 with parameters $q_{H}+1$ and $T_{H}'$. Recalling the $\mathrm{\mathbb{E}}[C]$
expressions in Table \ref{t1}, we note that, by assumption,
\begin{eqnarray}
\mathrm{\mathbb{E}}[X_{q_{H}}]=\mathrm{\mathbb{E}}[Y_{q_{H}+1}],\label{HP1-HP1}
\end{eqnarray}
where $\text{X\ensuremath{\sim}Poisson}(\lambda T_{H})$, $Y\sim\text{Poisson}(\lambda T_{H}')$.
Clearly, $T_{H}>T_{H}'$.

Next, recalling the results in Table \ref{t1} and the assumption
of fixed $\mathrm{\mathbb{E}}[C]$ values for all the policies of
interest, we need to show that
\begin{eqnarray}
\mathrm{\mathbb{E}}[X_{q_{H}}(X_{q_{H}}-1)]<\mathrm{\mathbb{E}}[Y_{q_{H}+1}(Y_{q_{H}+1}-1)].\label{HP1 vs HP1}
\end{eqnarray}
From Lemma \ref{E2<E2} and recalling  Eq. (\ref{HP1-HP1}), we have
\[
\mathrm{\mathbb{E}}[X_{q_{H}}^{2}]\leq\mathrm{\mathbb{E}}[Y_{q_{H}+1}^{2}],
\]
so that inequality (\ref{HP1 vs HP1}) is verified.

The same procedure can be applied to prove the similar result between
two HP2 policies.
\end{proof}

\section{Comparison of $AOD$ under Fixed Parameters $q$ and/or $T$\label{sec:Comparison-of-AOD-under fixed q and T}}

In \cite{CMW14}, the authors analytically show that under fixed parameters,
the general class of HPs outperform the general classes of counterpart
QP and TPs in terms of $AOD$. In this section, we provide a simpler
proof of the above statement based on the rewritten expressions in
Table \ref{t1}. Further, we show under fixed parameters, HP1 outperforms
HP2 in terms of $AOD$. The readers who are interested in the managerial
implications of this comparison of $AOD$ are referred to \cite{CMW14}.
\begin{thm}
\label{HP1bestfixed}With fixed parameters $q,T$, HP1 performs better
than QP and TP1 in terms of $AOD$.
\end{thm}

\begin{proof}
On one hand, we need to show HP1 performs better than QP in terms
of $AOD$ with the same parameters $q,T$, from Table \ref{t1}, that
is,
\begin{eqnarray*}
\frac{\mathrm{\mathbb{E}}[Y_{q}(Y_{q}-1)]}{\mathrm{\mathbb{E}}[Y_{q}]}<q-1.
\end{eqnarray*}
In fact, $(q-1)\mathrm{\mathbb{E}}[Y_{q}]-\mathrm{\mathbb{E}}[Y_{q}(Y_{q}-1)]=q\mathrm{\mathbb{E}}[Y_{q}]-\mathrm{\mathbb{E}}[Y_{q}^{2}]=\mathrm{\mathbb{E}}[(q-Y_{q})Y_{q}]>0$.

On the other hand, we need to show HP1 performs better than TP1
in terms of $AOD$ with the same parameters $q,T$, from Table \ref{t1},
that is,
\begin{eqnarray*}
\frac{\mathrm{\mathbb{E}}[Y_{q}(Y_{q}-1)]}{\mathrm{\mathbb{E}}[Y_{q}]}<\lambda T.
\end{eqnarray*}
In fact, from Lemma \ref{VAR<E}, we have $\mathbb{VAR}[Y_{q}]<\mathrm{\mathbb{E}}[Y_{q}]$,
which can written as
\[
\mathrm{\mathbb{E}}[Y_{q}(Y_{q}-1)]<\mathrm{\mathbb{E}}^{2}[Y_{q}].
\]
Further, we have $\mathrm{\mathbb{E}}[Y_{q}]<\lambda T$ since $Y\sim\text{Poisson}(\lambda T)$.
Thus, we arrive at the desired inequality.
\end{proof}
\begin{thm}
\label{HP2bestfixed}With fixed parameters $q,T$, HP2 performs better
than QP and TP2 in terms of $AOD$.
\end{thm}

\begin{proof}
On one hand, we need to show HP2 performs better than QP in terms
of $AOD$ with the same parameters $q,T$, from Table \ref{t1}, that
is,
\begin{eqnarray*}
\frac{\mathrm{\mathbb{E}}[Y_{q-1}(Y_{q-1}+1)]}{\mathrm{\mathbb{E}}[1+Y_{q-1}]}<q-1.
\end{eqnarray*}
In fact, $(q-1)\mathrm{\mathbb{E}}[1+Y_{q-1}]-\mathrm{\mathbb{E}}[Y_{q-1}(Y_{q-1}+1)]=\mathrm{\mathbb{E}}[(q-1-Y_{q-1})(Y_{q-1}+1)]>0$.

On the other hand, we need to show HP2 performs better than TP2
in terms of $AOD$ with the same parameters $q,T$, from Table \ref{t1},
that is,
\begin{eqnarray*}
\frac{\mathrm{\mathbb{E}}[Y_{q-1}(Y_{q-1}+1)]}{\mathrm{\mathbb{E}}[1+Y_{q-1}]}<\frac{2\lambda T+\lambda^{2}T^{2}}{1+\lambda T}.
\end{eqnarray*}
In fact, from Lemma \ref{VAR<E}, we have $\mathbb{VAR}[Y_{q-1}]<\mathrm{\mathbb{E}}[Y_{q-1}]$,
which can written as
\[
\mathrm{\mathbb{E}}[Y_{q-1}(Y_{q-1}+1)]<\mathrm{\mathbb{E}}^{2}[Y_{q-1}+1]-1.
\]
Further, due to $\mathrm{\mathbb{E}}[Y_{q-1}]<\lambda T$, we have
\begin{eqnarray*}
\mathrm{\mathbb{E}}[Y_{q-1}+1]-\frac{1}{\mathrm{\mathbb{E}}[Y_{q-1}+1]}<(\lambda T+1)-\frac{1}{\lambda T+1}.
\end{eqnarray*}
Thus, we arrive at the desired inequality.
\end{proof}

The following result allows us to compare HP1 and HP2 with fixed parameters
$q$ and $T$, which relies on Lemma \ref{ratiobetween2and1moment}.

\begin{thm}
\label{HP1betterthanHP2fixed}With fixed parameters $q$ and $T$,
HP1 performs better than HP2 in terms of $AOD$.
\end{thm}

\begin{proof}
From Table \ref{t1}, we need to show
\[
\frac{\mathrm{\mathbb{E}}[Y_{q}(Y_{q}-1)]}{\mathrm{\mathbb{E}}[Y_{q}]}<\frac{\mathrm{\mathbb{E}}[Y_{q-1}(Y_{q-1}+1)]}{\mathrm{\mathbb{E}}[1+Y_{q-1}]}.
\]
After simplification, it suffices to show
\begin{eqnarray}
\frac{\mathrm{\mathbb{E}}[Y_{q}^{2}]}{\mathrm{\mathbb{E}}[Y_{q}]}<\frac{\mathrm{\mathbb{E}}[(Y_{q-1}+1)^{2}]}{\mathrm{\mathbb{E}}[Y_{q-1}+1]}.\label{HP1 vs HP2}
\end{eqnarray}

Note for $X\sim\text{Poisson}(\mu)$, we have
\begin{eqnarray*}
\frac{d}{d\mu}\mathrm{\mathbb{E}}[g(X)]=\mathrm{\mathbb{E}}[g(X+1)]-\mathrm{\mathbb{E}}[g(X)],
\end{eqnarray*}
for any appropriate function $g(x)$. Let $\mu=\lambda T$, $g_{1}(x)=\left(x\wedge q\right)^{2}$,
and $g_{2}(x)=x\wedge q$, for $x\geq0$, we have
\begin{eqnarray*}
\frac{d}{d\mu}\mathrm{\mathbb{E}}[Y_{q}^{2}]=\frac{d}{d\mu}\mathrm{\mathbb{E}}[g_{1}(Y)]= &  & \mathrm{\mathbb{E}}[g_{1}(Y+1)]-\mathrm{\mathbb{E}}[g_{1}(Y)]\\
= &  & \mathrm{\mathbb{E}}[(Y_{q-1}+1)^{2}]-\mathrm{\mathbb{E}}[Y_{q}^{2}],
\end{eqnarray*}
and
\begin{eqnarray*}
\frac{d}{d\mu}\mathrm{\mathbb{E}}[Y_{q}]=\frac{d}{d\mu}\mathrm{\mathbb{E}}[g_{2}(Y)]= &  & \mathrm{\mathbb{E}}[g_{2}(Y+1)]-\mathrm{\mathbb{E}}[g_{2}(Y)]\\
= &  & \mathrm{\mathbb{E}}[Y_{q-1}+1]-\mathrm{\mathbb{E}}[Y_{q}].
\end{eqnarray*}
Hence,
\begin{eqnarray*}
\frac{d}{d\mu}\frac{\mathrm{\mathbb{E}}[Y_{q}^{2}]}{\mathrm{\mathbb{E}}[Y_{q}]}= &  & \frac{(\mathrm{\mathbb{E}}[(Y_{q-1}+1)^{2}]-\mathrm{\mathbb{E}}[Y_{q}^{2}])\mathrm{\mathbb{E}}[Y_{q}]-\mathrm{\mathbb{E}}[Y_{q}^{2}](\mathrm{\mathbb{E}}[Y_{q-1}+1]-\mathrm{\mathbb{E}}[Y_{q}])}{\mathrm{\mathbb{E}}^{2}[Y_{q}]}\\
= &  & \frac{\mathrm{\mathbb{E}}[(Y_{q-1}+1)^{2}]\mathrm{\mathbb{E}}[Y_{q}]-\mathrm{\mathbb{E}}[Y_{q}^{2}]\mathrm{\mathbb{E}}[Y_{q-1}+1]}{\mathrm{\mathbb{E}}^{2}[Y_{q}]}.
\end{eqnarray*}
From Lemma \ref{ratiobetween2and1moment}, we know $\frac{d}{d\mu}\frac{\mathrm{\mathbb{E}}[Y_{q}^{2}]}{\mathrm{\mathbb{E}}[Y_{q}]}>0$,
thus
\begin{eqnarray*}
\mathrm{\mathbb{E}}[(Y_{q-1}+1)^{2}]\mathrm{\mathbb{E}}[Y_{q}]-\mathrm{\mathbb{E}}[Y_{q}^{2}]\mathrm{\mathbb{E}}[Y_{q-1}+1]>0,
\end{eqnarray*}
which implies inequality (\ref{HP1 vs HP2}) is satisfied.
\end{proof}

\section{Conclusion\label{sec:Conclusion}}

Motivated by applications in shipment consolidation, we first provide a new unified method to compute $AOD$ for any renewal-type clearing policy based on a martingale associated with the Poisson process and the martingale
stopping theorem. Our goal is to provide a complete analytical comparison of alternative clearing policies (of type QP, TP, and HP) in terms of $AOD$. 
Our proposed method lends itself for a generic analytical characterization of $AOD$, leading to a complete comparative analysis of the policies of interest. 
In particular, we demonstrate that, under a fixed expected consolidation cycle length, QP outperforms any other renewal-type clearing policy in terms of $AOD$, not limited to HPs and TPs (see Theorem \ref{Qbest}).
Also, we complete the proof for the comparison of $AOD$ between HPs
and TPs under a fixed expected consolidation cycle length (see Theorems \ref{HPbetterthanTP} and \ref{revised}), and we provide
a simplified proof for the $AOD$ comparison among HPs, TPs and QP under
fixed parameters (see Theorems \ref{HP1bestfixed} and \ref{HP2bestfixed}), which are related to a property of truncated Poisson
random variables: for a truncated Poisson random variable $Y_{N}=\min(Y,N)$, $\mathbb{VAR}[Y_{N}]<\mathrm{\mathbb{E}}[Y_{N}]$ (see Lemma \ref{VAR<E}). 

Furthermore, we provide explicit and stronger comparative results between two HPs of the same type under a fixed expected consolidation cycle length (see Theorem \ref{HPvsHPundersameEC}),
which rely on a property of truncated random variables: given two
integer valued random variables $X,Y$, $X$ is stochastically larger
than $Y$, if $\mathrm{\mathbb{E}}[X_{q}]=\mathrm{\mathbb{E}}[Y_{q+1}]$,
where $q$ is a positive integer, then $\mathrm{\mathbb{E}}[X_{q}^{2}]\leq\mathrm{\mathbb{E}}[Y_{q+1}^{2}]$ (see Lemma \ref{E2<E2}).

Last but not least, we analytically show HP1 performs better than HP2 in terms
of $AOD$ under fixed parameters (see Theorem \ref{HP1betterthanHP2fixed}), which is equivalent to another property
of truncated Poisson random variables: $X\sim\text{Poisson}(\mu)$,
then $\mathrm{\mathbb{E}}[X_{N}^{2}]/\mathrm{\mathbb{E}}[X_{N}]$
is increasing with respect to $\mu$ (see Lemma \ref{ratiobetween2and1moment}).

Our results offer insightful and analytically justifiable guidance for logistics managers in selecting an appropriate shipment consolidation policy with an eye on the resulting service performance. Several challenging extensions of the problem at hand remain open for future research including the case where shipment consolidation efforts are subject to multiple and/or more general input processes, e.g., the case of a multi-commodity stochastic clearing system subject to Markov-modulated, renewal, or Brownian motion input processes.

\bibliographystyle{plain}
\bibliography{References}

\appendix

\section{Appendix: The Case with No Empty Shipments \label{sec:Append}}

Under TP1 and HP1, there may be empty shipments, which happens when
$N(T)=0$. In this Appendix, we consider revised TP1 and revised HP1,
which do not allow empty shipments, as introduced in Section \ref{sec:Introduction}.

Under revised HP1 with parameters $q,T$, the following recursion
equation about the expected consolidation cycle length $\mathrm{\mathbb{E}}[C_{RHP1}]$
is satisfied,
\begin{eqnarray}
\mathrm{\mathbb{E}}[C_{RHP1}]=\mathbb{P}(N(T)\geq1)\mathrm{\mathbb{E}}[\tau_{q}\wedge T\mid N(T)\geq1]+\mathbb{P}(N(T)=0)(T+\mathrm{\mathbb{E}}[C_{RHP1}]).\label{L-HP3-equation}
\end{eqnarray}
The equation means if no order arrives within $T$ units time, which
happens with probability $\mathbb{P}(N(T)=0)$, the consolidation
cycle restarts; if there are orders arriving within $T$ units time,
which happens with probability $\mathbb{P}(N(T)\geq1)$, the load
is dispatched at stopping time $\tau_{q}\wedge T$.

By noticing
\begin{eqnarray*}
 &  & \mathrm{\mathbb{E}}[\tau_{q}\wedge T]\\
= &  & \mathbb{P}(N(T)\geq1)\mathrm{\mathbb{E}}[\tau_{q}\wedge T\mid N(T)\geq1]+\mathbb{P}(N(T)=0)\mathrm{\mathbb{E}}[\tau_{q}\wedge T\mid N(T)=0]\\
= &  & \mathbb{P}(N(T)\geq1)\mathrm{\mathbb{E}}[\tau_{q}\wedge T\mid N(T)\geq1]+\mathbb{P}(N(T)=0)T,
\end{eqnarray*}
we have
\begin{eqnarray}
\mathbb{P}(N(T)\geq1)\mathrm{\mathbb{E}}[\tau_{q}\wedge T\mid N(T)\geq1]=\mathrm{\mathbb{E}}[\tau_{q}\wedge T]-\mathbb{P}(N(T)=0)T.\label{a}
\end{eqnarray}
Replacing Eq. (\ref{a}) into Eq. (\ref{L-HP3-equation}), and recalling $\mathrm{\mathbb{E}}[C_{HP1}]$
in Table \ref{t1}, we have
\begin{eqnarray}
\mathrm{\mathbb{E}}[C_{RHP1}]=\frac{\mathrm{\mathbb{E}}[\tau_{q}\wedge T]}{1-\mathbb{P}(N(T)=0)}=\frac{1}{\lambda}\frac{\mathrm{\mathbb{E}}[Y_{q}]}{1-\mathbb{P}(Y=0)}=\frac{\mathrm{\mathbb{E}}[C_{HP1}]}{1-\mathbb{P}(Y=0)},\label{firstL-HP3}
\end{eqnarray}
where $Y\sim\text{Poisson}(\lambda T)$.

Next, we calculate the expected cumulative delay within one consolidation
cycle under revised HP1, which is denoted as $\mathrm{\mathbb{E}}[W_{RHP1}]$.
The following recursion equation is satisfied,
\begin{eqnarray}
\mathrm{\mathbb{E}}[W_{RHP1}]=\mathbb{P}(N(T)\geq1)\mathrm{\mathbb{E}}[\int_{0}^{\tau_{q}\wedge T}N(t)dt|N(T)\geq1]+\mathbb{P}(N(T)=0)\mathrm{\mathbb{E}}[W_{RHP1}].\label{W-HP3-equation}
\end{eqnarray}
This equation means if no order arrives within $T$ units time, which
happens with probability $\mathbb{P}(N(T)=0)$, the consolidation
system restarts; if there are orders arriving within $T$ units time,
which happens with probability $\mathbb{P}(N(T)\geq1)$, the cumulative
delay within one consolidation cycle is $\int_{0}^{\tau_{q}\wedge T}N(t)dt$.

By noticing
\begin{eqnarray}
 &  & \mathrm{\mathbb{E}}[\int_{0}^{\tau_{q}\wedge T}N(t)dt]\nonumber \\
= &  & \mathbb{P}(N(T)\geq1)\mathrm{\mathbb{E}}[\int_{0}^{\tau_{q}\wedge T}N(t)dt|N(T)\geq1]+\mathbb{P}(N(T)=0)\mathrm{\mathbb{E}}[\int_{0}^{\tau_{q}\wedge T}N(t)dt|N(T)=0]\nonumber \\
= &  & \mathbb{P}(N(T)\geq1)\mathrm{\mathbb{E}}[\int_{0}^{\tau_{q}\wedge T}N(t)dt|N(T)\geq1],\label{b}
\end{eqnarray}
and replacing Eq. (\ref{b}) into Eq. (\ref{W-HP3-equation}), together with
recalling $\mathrm{\mathbb{E}}[W_{HP1}]$ in Table \ref{t1}, we have
\begin{eqnarray}
\mathrm{\mathbb{E}}[W_{RHP1}]=\frac{\mathrm{\mathbb{E}}[\int_{0}^{\tau_{q}\wedge T}N(t)dt]}{1-\mathbb{P}(N(T)=0)}=\frac{\mathrm{\mathbb{E}}[W_{HP1}]}{1-\mathbb{P}(Y=0)}=\frac{1}{2\lambda}\frac{\mathrm{\mathbb{E}}[Y_{q}(Y_{q}-1)]}{1-\mathbb{P}(Y=0)}.\label{firstW-HP3}
\end{eqnarray}

Define a new random variable $\tilde{Y}$, which has the same distribution
of $Y\mid Y>0$. We can rewrite
\begin{eqnarray}
\mathrm{\mathbb{E}}[C_{RHP1}]=\frac{1}{\lambda}\mathrm{\mathbb{E}}[\tilde{Y}_{q}],\label{L-HP3}
\end{eqnarray}
\begin{eqnarray}
\mathrm{\mathbb{E}}[W_{RHP1}]=\frac{1}{2\lambda}\mathrm{\mathbb{E}}[\tilde{Y}_{q}(\tilde{Y}_{q}-1)].\label{W-HP3}
\end{eqnarray}

Similarly, we can obtain the expected cycle length under revised TP1
with parameters $T$ is
\begin{eqnarray}
\mathrm{\mathbb{E}}[C_{RTP1}]=\frac{\mathrm{\mathbb{E}}[C_{TP1}]}{1-\mathbb{P}(N(T)=0)}=\frac{T}{1-e^{-\lambda T}},\label{L-TP3}
\end{eqnarray}
and the cumulative delay with one consolidation cycle under revised
TP1 with parameters $T$ is
\begin{eqnarray}
\mathrm{\mathbb{E}}[W_{RTP1}]=\frac{\mathrm{\mathbb{E}}[W_{TP1}]}{1-\mathbb{P}(N(T)=0)}=\frac{\lambda T^{2}}{2(1-e^{-\lambda T})}.\label{W-TP3}
\end{eqnarray}

From Eqs. (\ref{firstL-HP3}), (\ref{firstW-HP3}), (\ref{L-TP3}) and
(\ref{W-TP3}) and the definition of $AOD$, we have that $AOD$ of revised
HP1 is the same as HP1, $AOD$ of revised TP1 is the same as TP1, if
the parameters $q,T$ are fixed. From Theorem \ref{HP1bestfixed},
with fixed parameters $q,T$, revised HP1 also performs better than
QP and revised TP1 in terms of $AOD$.

From Theorem \ref{Qbest}, we can conclude that for a given expected
consolidation cycle length, QP performs better than revised HP1 and
revised TP1, in terms of $AOD$. In the following, we provide the comparison
between revised HP1 and revised TP1 with a given expected consolidation
cycle length.

Suppose $Y\sim\text{Poisson}(\lambda_{1})$, $Z\sim\text{Poisson}(\lambda_{2})$
and $\lambda_{1}>\lambda_{2}$, we know $Y$ is stochastically larger
than $Z$. Define $\tilde{Y}\sim Y\mid Y>0$, and $\tilde{Z}\sim Z\mid Z>0$,
we show $\tilde{Y}$ is also stochastically larger than $\tilde{Z}$
in the following lemma.
\begin{lem}
\label{largerconditional}Let $Y\sim\text{Poisson}(\lambda)$, $\tilde{Y}$
is distributed as $Y\mid Y>0$, then $\mathbb{P}(\tilde{Y}>n)$ is
increasing in $\lambda$, for any integer $n\geq1$.
\end{lem}

\begin{proof}
Notice $\frac{d}{d\lambda}\mathbb{P}(Y>n)=\mathbb{P}(Y=n)$. Then
for $n\geq1$,
\begin{eqnarray*}
\frac{d}{d\lambda}\mathbb{P}(\tilde{Y}>n)= &  & \frac{d}{d\lambda}\frac{\mathbb{P}(Y>n)}{\mathbb{P}(Y>0)}\\
= &  & \frac{\mathbb{P}(Y=n)\mathbb{P}(Y>0)-\mathbb{P}(Y>n)\mathbb{P}(Y=0)}{(\mathbb{P}(Y>0))^{2}}\\
= &  & \frac{\mathbb{P}(Y=n)-e^{-\lambda}\mathbb{P}(Y\geq n)}{(\mathbb{P}(Y>0))^{2}}.
\end{eqnarray*}
In addition, by using $\mathbb{P}(Y=k)=\frac{\lambda}{k}\mathbb{P}(Y=k-1)$,
we have
\begin{eqnarray*}
\mathbb{P}(Y=n)-e^{-\lambda}\mathbb{P}(Y\geq n)= &  & \frac{\lambda}{n}\mathbb{P}(Y=n-1)-e^{-\lambda}\sum_{k=n}^{\infty}\frac{\lambda}{k}\mathbb{P}(Y=k-1)\\
> &  & \frac{\lambda}{n}\left(\mathbb{P}(Y=n-1)-e^{-\lambda}\mathbb{P}(Y\geq n-1)\right).
\end{eqnarray*}
Since $\mathbb{P}(Y=0)-e^{-\lambda}\mathbb{P}(Y\geq0)=0$, it follows
by induction that
\begin{eqnarray*}
\mathbb{P}(Y=n)-e^{-\lambda}\mathbb{P}(Y\geq n)>0.
\end{eqnarray*}
Therefore, $\frac{d}{d\lambda}\mathbb{P}(\tilde{Y}>n)>0$.
\end{proof}

\begin{thm}\label{revised}
For a given expected consolidation cycle length $\mathrm{\mathbb{E}}[C]$,
the revised HP1 with larger quantity parameter achieves larger $AOD$
than the revised HP1 with smaller quantity parameter, in terms of
$AOD$. In particular, revised HP1 achieves less $AOD$ than revised TP1, under a given expected consolidation cycle length $\mathrm{\mathbb{E}}[C]$.
\end{thm}

\begin{proof}
We consider a fixed $\mathrm{\mathbb{E}}[C]$ and use the following
notation for the corresponding policy parameters under this $\mathrm{\mathbb{E}}[C]$
value: a revised HP1 with parameters $q_{H}$ and $T_{H}$, the other
revised HP1 with parameters $q_{H}+1$ and $T_{H}'$. Recalling Eq. (\ref{L-HP3})
and by assumption that the two revised HP1 have the same expected
cycle length, we have,
\begin{eqnarray}
\mathrm{\mathbb{E}}[\tilde{U}_{q_{H}}]=\mathrm{\mathbb{E}}[\tilde{V}_{q_{H}+1}],\label{VS-HP3}
\end{eqnarray}
where $\tilde{U}$ is distributed as $U\mid U>0$, $U\sim\text{Poisson}(\lambda T_{H})$,
and $\tilde{V}$ is distributed as $V\mid V>0$, $V\sim\text{Poisson}(\lambda T_{H}')$.
Clearly, $T_{H}>T_{H}'$. From Lemma \ref{largerconditional}, $\tilde{U}$
is stochastically larger than $\tilde{V}$.

Next, recalling Eq. (\ref{W-HP3}) and reiterating the assumption of fixed
$\mathrm{\mathbb{E}}[C]$, we proceed to show that
\begin{eqnarray}
\mathrm{\mathbb{E}}[\tilde{U}_{q_{H}}(\tilde{U}_{q_{H}}-1)]\leq\mathrm{\mathbb{E}}[\tilde{V}_{q_{H}+1}(\tilde{V}_{q_{H}+1}-1)].\label{VS-W-HP3}
\end{eqnarray}
From Lemma \ref{E2<E2}, and recalling Eq. (\ref{VS-HP3}), we have
\[
\mathrm{\mathbb{E}}[\tilde{U}_{q_{H}}^{2}]\leq\mathrm{\mathbb{E}}[\tilde{V}_{q_{H}+1}^{2}],
\]
so that inequality (\ref{VS-W-HP3}) is verified.

Revised TP1 can be considered as revised HP1 with quantity parameter $\infty$. Therefore, under the same expected consolidation cycle $\mathrm{\mathbb{E}}[C]$,
revised HP1 achieves less $AOD$ than revised TP1.
\end{proof}

\end{document}